\documentclass[a4paper, 11pt]{amsart}

\usepackage{graphicx, rotating}
\usepackage{subfigure}
\usepackage[justification=centering]{caption}

\usepackage[utf8]{inputenc}
\usepackage[toc]{appendix}
\usepackage{twoopt}

\usepackage{amssymb}
\usepackage{amsmath,a4wide}
\numberwithin{equation}{section}
\usepackage{hyperref}
\usepackage{amsthm}
\usepackage{mathtools}
\usepackage{amsfonts}
\usepackage{amsaddr}

\usepackage{marvosym}
\usepackage{verbatim}

\usepackage{enumerate}
\usepackage{pdfsync}
\usepackage{caption}
\usepackage{hyperref}
\usepackage{enumerate}
\mathtoolsset{showonlyrefs}

\theoremstyle{plain}
\newtheorem{theorem}{Theorem}[section]
\theoremstyle{plain}

\theoremstyle{plain}

\theoremstyle{plain}
\newtheorem{proposition}[theorem]{Proposition}
\theoremstyle{definition}
\newtheorem{definition}[theorem]{Definition}
\theoremstyle{remark}

\theoremstyle{remark}

\theoremstyle{definition}

\theoremstyle{plain}
\newtheorem{conjecture}[theorem]{Conjecture}
\theoremstyle{definition}

\providecommand{\norm}[1]{\left\lVert #1 \right\rVert}

\newcommand{\R}{\mathbb{R}}
\newcommand{\C}{\mathbb{C}}
\newcommand{\Rd}{\mathbb{R}^d}
\newcommand{\Z}{\mathbb{Z}}
\newcommand{\N}{\mathbb{N}}
\newcommand{\T}{\mathbb{T}}
\newcommand{\Lt}[1][d]{L^2(\R^{#1})}
\newcommand{\G}{\mathcal{G}}
\newcommand{\F}{\mathcal{F}}

\newcommand{\indicator}{\raisebox{2pt}{$\chi$}}
\renewcommand{\l}{\lambda}
\renewcommand{\L}{\Lambda}

\renewcommand{\H}{\mathbb{H}}
\newcommand{\vol}{\textnormal{vol}}

\newcommand{\sump}[1]{{\sum_{#1}}^\prime}
\newcommand{\detp}{{\det}^\prime}

\DeclareMathOperator*{\essinf}{ess\,inf}
\DeclareMathOperator*{\esssup}{ess\,sup}

\renewcommand{\S}
	{
		\left(
		\begin{array}{c c}
			a & b\\
			c & d
		\end{array}	
		\right)	
	}

\newcommand{\J}
	{
		\left(
		\begin{array}{r c}
			0 & 1\\
			-1 & 0
		\end{array}	
		\right)	
	}

\newcommandtwoopt{\xarrow}[2][0.5cm][0]{\mathrel{\rotatebox[origin=c]{#2}{$\xrightarrow{\rule{#1}{0pt}}$}}}

\begin{document}

\title[Some curious results related to a conjecture of Strohmer and Beaver]{
		Some curious results related to a conjecture of Strohmer and Beaver
	}

\author[Markus Faulhuber]{Markus Faulhuber}
\address{NuHAG, Fcaulty of Mathematics, University of Vienna\\Oskar-Morgenstern-Platz 1, 1090 Vienna, Austria}
\address{Department of Mathematics, RWTH Aachen University\\Schinkelstraße 2, 52062
Aachen, Germany
}
\email{markus.faulhuber@univie.ac.at}
\thanks{The author was supported in part by the Vienna Science and Technology Fund (WWTF): VRG12-009 and in part by an Erwin--Schrödinger Fellowship of the Austrian Science Fund (FWF): J4100-N32. The author spent the first stage of the fellowship with the Analysis Group at NTNU Trondheim where parts of this work have been established. The author wishes to thank Franz Luef for beneficial feedback on the first draft of the manuscript.}

\begin{abstract}
	We study results related to a conjecture formulated by Strohmer and Beaver about optimal Gaussian Gabor frame set-ups. Our attention will be restricted to the case of Gabor systems with standard Gaussian window and rectangular lattices of density 2. Although this case has been fully treated by Faulhuber and Steinerberger, the results in this work are new and quite curious. Indeed, the optimality of the square lattice for the Tolimieri and Orr bound already implies the optimality of the square lattice for the sharp lower frame bound. Our main tools include determinants of Laplace--Beltrami operators on tori as well as special functions from analytic number theory, in particular Eisenstein series, zeta functions, theta functions and Kronecker's limit formula. We note that our results also carry over to energy minimization problems over lattices and a heat distribution problem over flat tori.
\end{abstract}


\maketitle

\section{Introduction}\label{sec_Intro}
The conjecture of Thomas Strohmer and Scott Beaver \cite{StrBea03} relates to the problem of finding an optimal sampling strategy for Gaussian Gabor systems. They conjecture that a hexagonal pattern in the time-frequency plane should be used in order to minimize the condition number of the Gabor frame operator. This can also be formulated as an engineering problem about optimal orthogonal frequency division multiplexing (OFDM) systems and lattice-OFDM systems (see \cite{StrBea03}). We will use the Gabor frame formulation of the problem.

A Gabor system (see Section \ref{sec_TFA}) is a collection of functions obtained from combined translations $T_x$ and modulations $M_\omega$ of a so-called window function $g \in \Lt[]$. The translation parameters $x$ and modulation parameters $\omega$ are chosen with respect to an index set in $\R^2$. Throughout this work the index pairs $(x, \omega)$ will generate lattices in $\R^2$. A Gabor system is a frame for $\Lt[]$, if any element of the Hilbert space has a stable series expansion with respect to the Gabor system. These systems have wide applications and are frequently used in signal analysis and processing and also found their way into the popular area of machine learning. There, Gabor systems are for example used to pre-process audio data in order to extract specific features which can then be passed on to a convolutional neural network (see, e.g., \cite{BamDorHar19}, \cite{MarHolPerMaj19}, \cite{MarPerHolMaj19}).

Given a standard Gaussian window, we are looking for a lattice of given density (greater than 1) such that the condition number of the associated Gabor frame operator is minimal. For general lattices it is conjectured that the hexagonal lattice is optimal. If only separable (or rectangular) lattices are considered, the conjecture is that the square lattice provides the optimal solution. For even density, it was proven that the hexagonal lattice uniquely minimizes the upper frame bound \cite{Faulhuber_Hexagonal_2018} and in the separable case the square lattice minimizes the upper frame bound and maximizes the lower frame bound, implying that the condition number is minimal \cite{FaulhuberSteinerberger_Theta_2017}. 

From a qualitative point of view, the lower frame bound and the upper frame bound describe different aspects of the frame operator. Whereas the lower bound concerns the (bounded) invertibility of the frame operator, the upper bound gives the continuity of the operator since the frame operator is linear. Therefore, our main result can be considered to be quite curious.

\begin{theorem}[Main Result]\label{thm_main}
	Let $g_0(t) = 2^{1/4} e^{- \pi t^2}$ be the normalized standard Gaussian and $\L_{(\alpha, \beta)} = \alpha \Z \times \beta \Z$. We denote the optimal lower and upper frame bound of the Gabor system $\G(g_0, \L_{(\alpha,\beta)})$ by $A(\alpha,\beta)$ and $B(\alpha, \beta)$ respectively. By $\widetilde{B}(\alpha,\beta)$ we denote the (re-normalized) Tolimieri and Orr bound
	\begin{equation}
		\widetilde{B}(\alpha,\beta) = \sum_{\l^\circ \in \L_{(\alpha, \beta)}^\circ} |V_{g_0} g_0(\l^\circ)|, \qquad \L_{(\alpha, \beta)}^\circ = \tfrac{1}{\beta} \Z \times \tfrac{1}{\alpha} \Z.
	\end{equation}
	Then, for $(\alpha \beta)^{-1} = 2$, we have the following implications
	\begin{align}
		\widetilde{B} \left( \tfrac{t}{\sqrt{2}}, \tfrac{t}{\sqrt{2}}\right) \leq \widetilde{B} (t \alpha, t \beta),
		\quad \forall t \in \R_+ \quad
		& \Longrightarrow \quad
		A \left( \tfrac{1}{\sqrt{2}}, \tfrac{1}{\sqrt{2}}\right) \geq A (\alpha, \beta).
	\end{align}
	In each case, equality holds only if $\alpha = \beta = \tfrac{1}{\sqrt{2}}$.
\end{theorem}

We note that, for $(\alpha \beta)^{-1} = 2$, the results
\begin{equation}
	\widetilde{B} \left( \tfrac{t}{\sqrt{2}}, \tfrac{t}{\sqrt{2}}\right) \leq \widetilde{B} (t \alpha, t \beta),	\quad \forall t \in \R_+
	\qquad \textnormal{ and } \qquad
	A \left( \tfrac{1}{\sqrt{2}}, \tfrac{1}{\sqrt{2}}\right) \geq A (\alpha, \beta),
\end{equation}
have been established already in \cite{FaulhuberSteinerberger_Theta_2017} and the result on $\widetilde{B}$ is actually covered by results of Montgomery\cite{Montgomery_Theta_1988}. The new insight of the above theorem is the implication
\begin{equation}
	\widetilde{B} \left( \tfrac{t}{\sqrt{2}}, \tfrac{t}{\sqrt{2}}\right) \leq \widetilde{B} (t \alpha, t \beta),	\quad \forall t \in \R_+
	\qquad \Longrightarrow \qquad
	A \left( \tfrac{1}{\sqrt{2}}, \tfrac{1}{\sqrt{2}}\right) \geq A (\alpha, \beta).
\end{equation}

To obtain the implication in Theorem \ref{thm_main}, we need to exploit properties of Laplace-Beltrami operators on tori, their determinants and connections to zeta functions as well as other tools from analytic number theory. The use of this machinery is definitely non-standard in the theory of Gabor systems. In addition, the proof relies on the separability of the lattice and the author was not able to establish similar implications for non-separable lattices, even though it is known, due to the results in \cite{Faulhuber_Hexagonal_2018} and \cite{Montgomery_Theta_1988}, that the hexagonal lattice always gives the smallest Tolimieri and Orr bound for fixed lattice density.

Even though this work deals with a problem in time-frequency analysis, the results are of a more general nature. We note that, for example, Montgomery's main result in \cite{Montgomery_Theta_1988} shows the so-called universal optimality of the hexagonal lattice for the problem of lattice energy minimization as most recently described in \cite{Coh-Via19}. Now, minimizing $\widetilde{B}$ among lattices of fixed density is equivalent to the problem of determining the universal optimality of the hexagonal lattice and the problem considered by Montgomery \cite{Montgomery_Theta_1988}, as shown in \cite{Faulhuber_Hexagonal_2018}. The problem for general point sets in $\R^2$ is still open and its solution is of utmost importance. We refer to \cite{Coh-Via19} for further reading.

On the other hand, the problem of determining the largest possible lower frame bound is intimately connected to a heat distribution problem on flat tori as described by Baernstein \cite{Baernstein_HeatKernel_1997}, who also conjectures a connection (see also \cite{Baernstein_ExtremalProblems_1994} and \cite{BaernsteinVinson_Local_1998}) to a long-standing problem of Landau \cite{Lan29}. A new connection between the heat kernel problem and Landau's problem was recently established in \cite{Faulhuber_Rama_2019}. We will draw a connection to the heat kernel problem in Section \ref{sec_discuss}. For further reading on the connections between Gabor systems and Landau's problem we also refer to \cite{Faulhuber_SampTA19}.

This work is structured as follows.
\begin{itemize}
	\item In Section \ref{sec_TFA} we introduce Gabor systems and frames for the Hilbert space $\Lt[]$. The concepts can easily be generalized to higher dimensions, but in this work we only consider the 1-dimensional case. At the end of the section we state 2 conjectures resulting from a conjecture of Strohmer and Beaver \cite{StrBea03} and the results of Faulhuber \cite{Faulhuber_Hexagonal_2018} as well as Faulhuber and Steinerberger \cite{FaulhuberSteinerberger_Theta_2017}.
	\item In Section \ref{sec_Janssen} we introduce some methods, going back to the work of Janssen \cite{Jan95}, \cite{Jan96}, which allow us to compute sharp frame bounds.
	\item In Section \ref{sec_NT} we leave the field of time-frequency analysis and introduce some tools from analytic number theory. These tools include theta functions, zeta functions related to the Laplace--Beltrami operator on a torus and the Dedekind eta function. Equipped with these tools, we will prove our main result.
	\item In Section \ref{sec_discuss} we discuss some aspects of the proof and draw connections to the heat distribution problem on flat tori.
	\item In Appendix \ref{sec_symplectic}, we briefly explain what the symplectic Fourier transform and the symplectic Poisson summation formula are. We will use them as they fit nicely with the concepts of time-frequency analysis.
\end{itemize}

\section{Time-Frequency Analysis}\label{sec_TFA}
\subsection{Gabor Systems}

In this section we will set the notation and briefly recall the notion of a Gabor system. We are interested in Gabor systems for square integrable functions on the line. Hence, the Hilbert space of interest is $\Lt[]$. We note that the concepts can easily be generalized to higher dimensions or, more general, to locally compact Abelian groups (see e.g.~\cite{JakLem_Density_2016}). For an introduction or more information on Gabor systems and frames we refer to some of the standard literature e.g.~\cite{Christensen_2016}, \cite{FeiStr98}, \cite{FeiStr03}, \cite{Gro01}, \cite{Hei06} and for a more current state of the research and important open problems in Gabor analysis we refer e.g.~to the expository articles \cite{Fei_ATFA17} and \cite{Gro14}.

For the inner product in $\Lt[]$ we write
\begin{equation}
	\langle f, g\rangle = \int_\R f(t) \, \overline{g(t)} \, dt,
\end{equation}
where $\overline{g}$ denotes the complex conjugate of $g$. For the Fourier transform we write
\begin{equation}
	\F f(\omega) = \int_\R f(t) e^{-2 \pi i \omega t} \, dt.
\end{equation}
Next, we introduce the fundamental tool in time-frequency analysis, the short-time Fourier transform.
\begin{definition}[STFT]
	For a function $f \in \Lt[]$ and a window $g \in \Lt[]$, the short-time Fourier transform of $f$ with respect to $g$ is defined as
	\begin{equation}
		V_g f(x, \omega) = \int_\R f(t) \overline{g(t-x)} e^{-2 \pi i \omega t} \, dt , \qquad x, \omega \in \R.
	\end{equation}
\end{definition}

A Gabor system for $\Lt[]$ is generated by a (fixed, non-zero) window function $g \in \Lt[]$ and an index set $\L \subset \R^2$ and is denoted by $\G(g,\L)$. These systems were studied already by von Neumann in the context of quantum mechanics \cite{Neumann_Quantenmechanik_1932} and became popular in engineering due to the article by Gabor \cite{Gab46}. A Gabor system consists of time-frequency shifted versions of $g$ and the elements $\pi(\l) g$, which will introduced now, are called atoms. By $\l = (x, \omega) \in \R^2$ we denote the generic point in the time-frequency plane\footnote{We note that $\R$ is a locally Abelian group and that its dual group is $\widehat{\R}$, the group of characters. The time-frequency plane $\R^2$ is actually isomorphic to $\R \times \widehat{\R}$.} and for a time-frequency shift by $\l$ we write
\begin{equation}
	\pi (\l)g(t) = M_\omega T_x \, g(t) = e^{2 \pi i \omega t} g(t-x), \quad x,\omega,t \in \R.
\end{equation}
In general, time-frequency shifts do not commute, as already the time-shift and frequency-shift operators do not commute;
\begin{equation}
	M_\omega T_x = e^{2 \pi i \omega x} T_x M_\omega.
\end{equation}
This results in the following commutation relations for time-frequency shifts;
\begin{equation}\label{eq_comm_rel}
	\pi(\l)\pi(\l') = e^{2 \pi i (x \omega' - x' \omega)} \pi(\l')\pi(\l).
\end{equation}
These commutation relations are the very reason for the importance of the symplectic group in time-frequency analysis and why we will use the symplectic Fourier transform and a symplectic version of the Poisson summation formula (see Appendix \ref{sec_symplectic}).

For a window function $g$ and an index set $\L$ the Gabor system is
\begin{equation}
	\G(g,\L) = \lbrace \pi(\l) g \, | \, \l \in \L \rbrace.
\end{equation}
We note that we have $V_g f(\l) = \langle f, \pi(\l) g \rangle$. Hence, sampling $V_gf$ on $\L$ is the same as taking linear measurement with respect to the atoms $\pi(\l)g$ of the Gabor system $\G(g,\L)$.

It is of particular interest to know whether a Gabor system is also a (Gabor) frame. In this case, the Gabor system $\G(g,\L)$ serves as a generalization of an orthnormal basis.
\begin{definition}[Gabor frame]
	In order to be a frame, $\G(g,\L)$ has to satisfy the frame inequality
	\begin{equation}\label{eq_frame}
		A \norm{f}_2^2 \leq \sum_{\l \in \L} \left| \langle f, \pi(\l) g \rangle \right|^2 \leq B \norm{f}_2^2, \quad \forall f \in \Lt[],
	\end{equation}
	for some positive constants $0 < A \leq B < \infty$ called frame constants or frame bounds.
\end{definition}
In this work, whenever we use the term frame bounds, we usually refer to the tightest possible (optimal) bounds in \eqref{eq_frame}. It is clear from the frame inequality \eqref{eq_frame} that the frame bounds crucially depend on the window $g$ as well as the index set $\L$ and in general it is hard to calculate sharp frame bounds. However, for certain windows and certain index sets we have explicit formulas, as we will see later.

A prototype example of a Gabor system, which also is a Gabor frame and actually constitutes an orthonormal basis for $\Lt$, is given by
\begin{equation}
	\G(\indicator_{[-\frac{1}{2}, \frac{1}{2}]}, \Z^2) = \{ e^{2 \pi i l t} \indicator_{[-\frac{1}{2}, \frac{1}{2}]}(t-k) \mid (k,l) \in \Z^2\}.
\end{equation}
Note that this example uses the well-known Fourier basis for $L^2([-\frac{1}{2}, \frac{1}{2}])$ (or $L^2$ of any other interval of length 1) and places (non-overlapping) copies of it on the real line. The frame inequality \eqref{eq_frame} actually becomes an equality with bounds $A=B=1$. Moreover, we can expand any $f \in \Lt[]$ with respect to the Gabor system as
\begin{align}
	f(t) & = \sum_{(k,l) \in \Z^2} \langle f, \pi(k,l) \indicator_{[-\frac{1}{2}, \frac{1}{2}]} \rangle \, \pi(k,l) \indicator_{[-\frac{1}{2}, \frac{1}{2}]}(t)\\
	& = \sum_{(k,l) \in \Z^2} c_{k,l} \, e^{2 \pi i l t} \, \indicator_{[-\frac{1}{2}, \frac{1}{2}]}(t-k)
\end{align}
The coefficients $c_{k,l}$ are obtained from the STFT, just as the coefficients in a Fourier series are obtained from the respective Fourier transform of the function on an interval. For general a window $g$ and index set $\L$, one is interested in expansions of the form
\begin{equation}
	f = \sum_{\l \in \L} c_\l \, \pi(\l) g,
\end{equation}
for all $f$ in $\Lt[]$ such that $(c_\l) \in \ell^2(\L)$, which is only possible if the frame inequality is fulfilled \cite{LyuSei99}.

Loosely speaking, the coefficients contain local information of the time-frequency content of a function $f$, in contrast to the Fourier transform $\F f(\omega)$, where we only know the total contribution of the frequency $\omega$ to all of $f$, but not to $f$ localized around some point $x$. In order to obtain valuable local information obtained from the coefficients $c_\l$, one needs a window $g$ which localizes $f$ well and where the Fourier transform $\F g$ localizes $\F f$ sufficiently well too. The localization properties of the indicator function $\indicator_{[-\frac{1}{2}, \frac{1}{2}]}(t)$ are therefore not satisfying, as its Fourier transform is the sinc function $\frac{\sin(\pi \omega)}{\pi \omega}$, which is slowly decaying. Now, under very mild assumptions on the window $g$, such as being continuous, and piece-wise differentiable and decaying faster than $\frac{1}{x}$, it is not possible to obtain an orthonormal basis from a Gabor system. This is known as the Balian-low theorem. Therefore, one needs to deal with overcomplete systems, which also means that, in general, the elements $\pi(\l)g$ of the Gabor system $\G(g,\L)$ are neither linearly independent nor orthogonal any more.

To the Gabor system $\G(g,\L)$ we can associate the Gabor frame operator $S_{g,\L}$ which acts on an element $f \in \Lt[]$ by the rule
\begin{equation}
	S_{g,\L} f = \sum_{\l \in \L} \langle f, \pi(\l) g \rangle \, \pi(\l) g.
\end{equation}
In general, $S_{g,\L} f$ and $f$ are different elements in $\Lt[]$. However, if $\G(g,\L)$ is an orthonormal basis, as in the example above, then $S_{g,\L}$ is the identity operator on $\Lt[]$. If $\G(g,\L)$ is a Gabor frame, then the frame operator is bounded and boundedly invertible and we can expand any $f \in \Lt[]$, e.g., as
\begin{equation}
	f = S_{g,\L}^{-1} S_{g,\L}f = \sum_{\l \in \L} \underbrace{\langle f, S_{g,\L}^{-1} \pi(\l) g\rangle}_{c_\l} \, \pi(\l) g.
\end{equation}
Interchanging the action of the inverse from operator and summation is justified as, under the assumption that $\G(g,\L)$ is a frame, the series expansion of $f$ obtained by the action of the frame operator is a Bessel series and converges unconditionally \cite[Chap.~5.3]{Gro01}.

The (optimal) frame bounds are connected to the frame operator in the following way
\begin{align}
	A^{-1} = \norm{S_{g,\L}^{-1}}_{op} \qquad \text{ and } \qquad B = \norm{S_{g,\L}}_{op}.
\end{align}
The condition number of the frame operator is
\begin{equation}
	cond(S_{g,\L}) = B/A.
\end{equation}
Hence, the frame bounds serve as a quantitative measure of how close or far the frame operator is from (a multiple of) the identity operator. Also, it is of interest for implementations to have a small condition number of the frame operator when transmitting data using an (N-)OFDM system implemented by a Gabor frame\footnote{As the elements of a Gabor frame are in general not orthogonal, it is appropriate to speak about non-orthogonal frequency division multiplexing (N-OFDM).}.

As mentioned in the introduction, qualitatively the lower bound shows that the frame operator is boundedly invertible and the upper bound shows that the frame operator is continuous on $\Lt[]$. Therefore, one may not expect that an optimality condition on one bound implies the optimality of the other bound.

\subsection{Lattices}
Throughout the rest of this work the index set $\L \subset \R^2$ will be a lattice. A lattice $\L$ is generated by an invertible (non-unique) $2 \times 2$ matrix $M$, in the sense that
\begin{equation}
	\L = M \Z^2 = \lbrace k v_1 + l v_2 \mid k,l \in \Z \rbrace,
\end{equation}
where $v_1$ and $v_2$ are the columns of $M$. The parallelogram spanned by the vectors $v_1$ and $v_2$ is called the fundamental domain of $\L$. We note that the matrix $M$ defining a lattice is not unique as we may choose from a countable set of bases for $\Z^2$. More precisely, any matrix $\mathcal{B} \in SL(2,\Z)$, i.e., any matrix of determinant 1 with integer entries leaves $\Z^2$ invariant;
\begin{equation}
	\mathcal{B} \Z^2 = \Z^2.
\end{equation}
Hence, $\L = M \Z^2 = (M \mathcal{B}) \Z^2$ for any $\mathcal{B} \in SL(2,\Z)$. We also refer to \cite[Chap.~VII]{Serre_Course_1973} for the importance of $SL(2,\Z)$ in the theory of modular forms, which will be discussed briefly in Section \ref{sec_modular}. 

The volume and the density of the lattice are given by
\begin{equation}
	\text{vol}(\L) = |\det(M)| \qquad \text{ and } \qquad \delta(\L) = \frac{1}{\text{vol}(\L)},
\end{equation}
respectively. For a Gabor system $\G(g,\L)$, the density $\delta(\L)$ determines the redundancy of the system. If $\delta < 1$, the system fails to be a frame as not enough information is contained in the coefficient sequence $(c_\l)$. The case $\delta = 1$ is referred to as the critical density and in this case orthonormal bases and frames exist. If $\delta > 1$, frames exist, but they cannot be an orthonormal basis. However, the condition $\delta \geq 1$ is only a necessary condition and not always sufficient for a Gabor system to constitute a frame. We refer to \cite[Chap.~7.5]{Gro01} or \cite{Hei07} for more information on density conditions for Gabor frames.

The adjoint lattice $\L^\circ$ is defined as
\begin{equation}
	\L^\circ = J M^{-T} \Z^2, \qquad J = \J
\end{equation}
and $M^{-T}$ denotes the transposed inverse of $M$. Hence, the adjoint lattice is just a 90 degrees rotated version of the dual lattice, usually denoted by $\L^\bot$. The formal, but equivalent, definition of the adjoint lattice is
\begin{equation}
	\L^\circ = \{ \l^\circ \in \R^2 \mid \pi(\l)\pi(\l^\circ) = \pi(\l^\circ) \pi(\l), \, \forall \l \in \L \}.
\end{equation}
Note that the adjoint lattice is characterized by the commutation relations \eqref{eq_comm_rel} and equivalently we have
\begin{equation}
	e^{2 \pi i \sigma(\l^\circ, \l)} = 1 \Longleftrightarrow \sigma(\l^\circ, \l) \in \Z,
\end{equation}
where $\sigma(.,.)$ is the standard symplectic form (see Appendix \ref{sec_symplectic}). This underlines the analogy of the adjoint lattice $\L^\circ$ to the dual lattice $\L^\perp$, which is characterized by the fact that
\begin{equation}
	e^{2 \pi i \l^\perp \cdot \l} = 1 \Longleftrightarrow \l^\perp \cdot \l \in \Z,
\end{equation}
where $\l^\perp \cdot \l$ is the Euclidean inner product between $\l^\perp$ and $\l$.

In the sequel, we will often limit our attention to rectangular or separable lattices. A lattice is called separable if the generating matrix can be represented by a diagonal matrix, which means that the lattice and its adjoint are given by
\begin{equation}
	\L_{(\alpha, \beta)} = \alpha \Z \times \beta \Z \qquad \textnormal{ and } \qquad
	\L_{(\alpha, \beta)}^\circ = \tfrac{1}{\beta} \Z \times \tfrac{1}{\alpha} \Z.
\end{equation}
Note that the dual lattice is given by $\L_{(\alpha, \beta)}^\perp = \frac{1}{\alpha} \Z \times \frac{1}{\beta} \Z$.

\subsection{Conjectures Related to the Work of Strohmer and Beaver}
Throughout this work we call the function $g_0(t) = 2^{1/4} e^{-\pi t^2}$ the standard Gaussian. In their article \cite{StrBea03} from 2003, Strohmer and Beaver conjecture that for any fixed density the hexagonal lattice minimizes the condition number $B/A$ of the Gabor frame operator $S_{g_0,\L}$. The following conjecture implies the conjecture of Strohmer and Beaver.
\begin{conjecture}\label{con_generalized_Strohmer_Beaver}
	For the standard Gaussian window and any fixed lattice density $\delta > 1$, among all lattices in the set
	\begin{equation}
		\mathfrak{F}_{\L}^\delta (g_0) = \lbrace \Lambda \subset \R^2 \, \textnormal{a lattice} \, | \, \textnormal{vol}(\Lambda)^{-1} = \delta \rbrace,
	\end{equation}
	the hexagonal lattice is the unique maximizer for the lower frame bound and the unique minimizer for the upper frame bound.
\end{conjecture}
Before the work of Strohmer and Beaver a conjecture by le Floch, Alard and Berrou (1995) \cite{FloAlaBer95} suggested that the square lattice could be optimal. This was disproved in \cite{StrBea03}, however, if we only consider the separable case we get the following conjecture.
\begin{conjecture}\label{con_generalized_Floch}
	For the standard Gaussian window and any fixed lattice density $\delta > 1$, among all separable lattices in the set
	\begin{equation}
		\mathfrak{F}_{(\alpha,\beta)}^\delta (g_0) = \lbrace \alpha \Z \times \beta \Z \mid \alpha, \beta \in \R_+, \, (\alpha \beta)^{-1} = \delta \rbrace,
	\end{equation}
	the square lattice is the unique maximizer for the lower frame bound and the unique minimizer for the upper frame bound.
\end{conjecture}
For integer densities of the lattice, Conjecture \ref{con_generalized_Floch} was proved by Faulhuber and Steinerberger \cite{FaulhuberSteinerberger_Theta_2017}. Also, for even integer densities of the lattice it was shown that the hexagonal lattice minimizes the upper frame bound \cite{Faulhuber_Hexagonal_2018}. The general case for the lower frame bound is still open for all possible densities\footnote{Tolimieri and Orr \cite{TolOrr95} mention that, for even densities, in private correspondence with Janssen they found a lengthy argument solving the problem for the upper frame bound in Conjecture \ref{con_generalized_Floch}. It is very likely that the argument they found is similar to the proof in \cite{FaulhuberSteinerberger_Theta_2017}.}.

\begin{figure}[ht]
	\includegraphics[width=.45\textwidth]{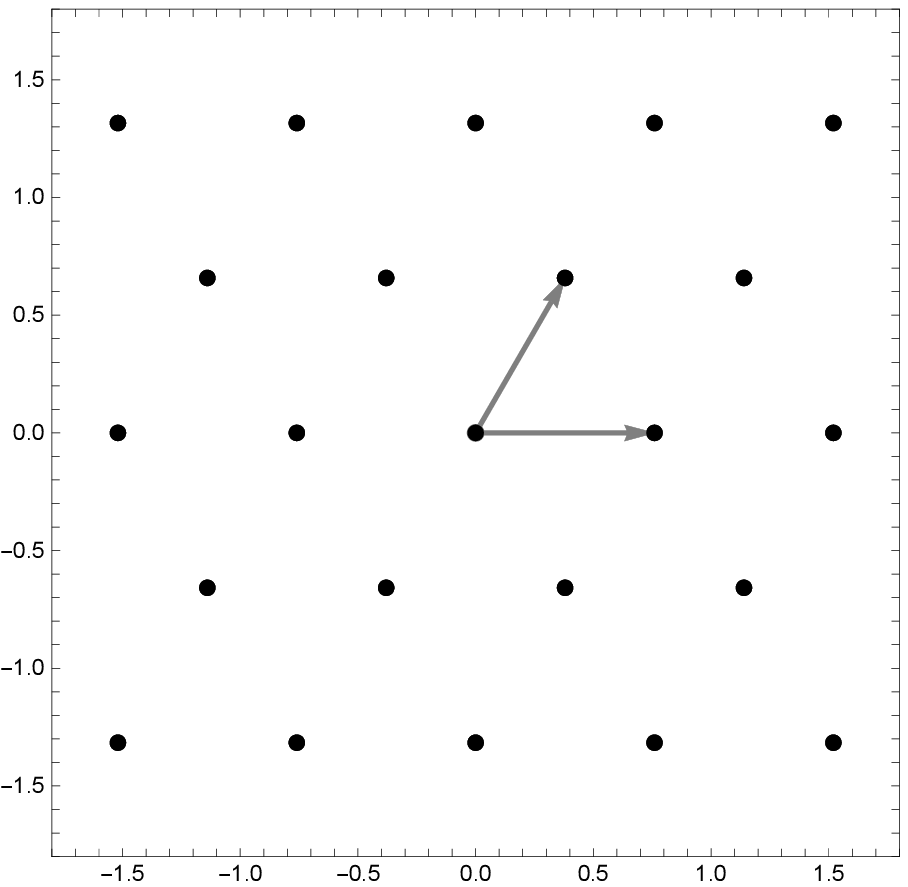}
	\hfill
	\includegraphics[width=.45\textwidth]{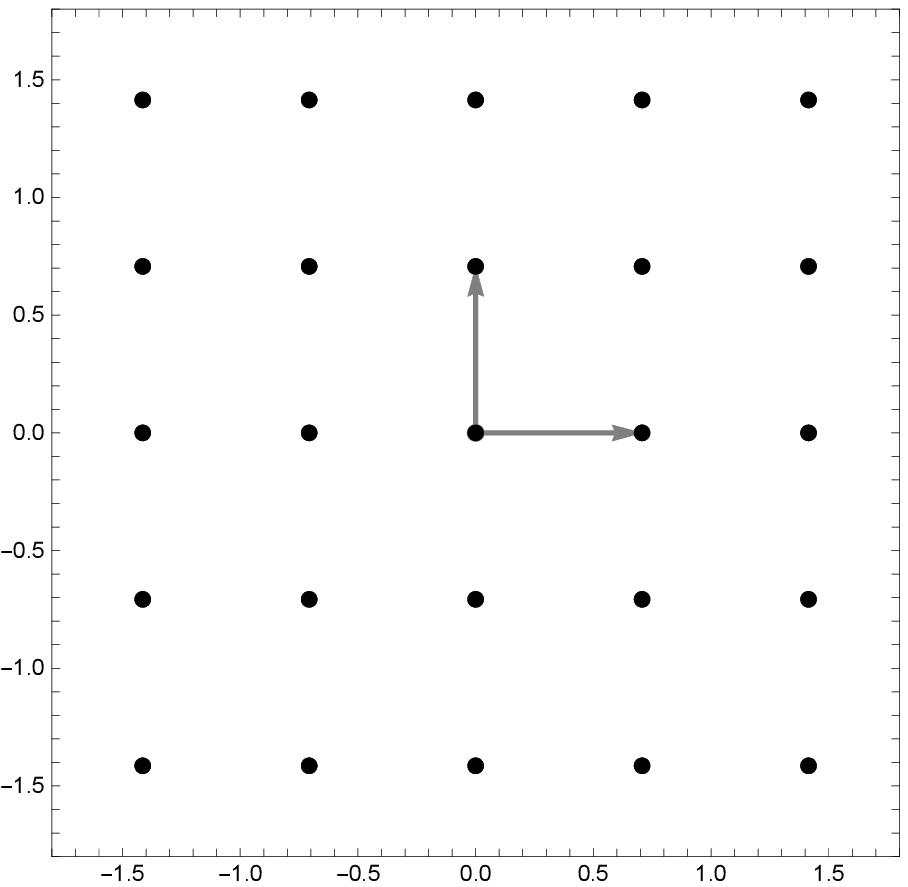}
	\caption{A hexagonal (sometimes called triangular) lattice and a square lattice of density 2. A possible basis is marked in both cases.}
\end{figure}

\section{Sharp Frame Bounds}\label{sec_Janssen}
\subsection{Janssen's Work}
Our starting point for proving Theorem \ref{thm_main} is the work of Janssen \cite{Jan95}, \cite{Jan96}. Janssen observed that for integer density the frame bounds can be computed exactly by finding the minimum and the maximum of certain Fourier series related to the frame operator. This follows from the duality theory which he described in \cite{Jan95} and which allows to compute the spectrum of the frame operator via the spectrum of a related, bi-infinite matrix. For integer density of the lattice this bi-infinite matrix has a Laurent structure and, hence, by the general theory of Toeplitz and Laurent matrices and operators, the spectrum can be calculated by using Fourier series with the matrix entries as coefficients \cite{Jan96}.

Janssen stated his results for $\Lt[]$ and separable lattices. However, since for $d=1$ any lattice is symplectic, this actually already covers all cases since any Gabor system can be transformed into a Gabor system with separable lattice without losing any of its properties by using a proper unitary operator from the so-called metaplectic group (see e.g.~\cite{Faulhuber_Invariance_2016}, \cite{Fol89}, \cite{Gos11}, \cite{Gos15}, \cite{Gro01}).

Let us now state the result in its original form, since we do not need a more general setting. A standard assumption is
\begin{equation}\label{eq_condA}
	\sum_{\l^\circ \in \L^\circ} \left| V_g g \left( \l^\circ \right) \right| < \infty. \tag{Condition A}
\end{equation}

As we will only need separable lattices in the sequel, we state the following result as given in \cite{Jan96}.
\begin{proposition}\label{pro_Janssen}
	Let $g \in \Lt[]$ and $\L_{(\alpha,\beta)} = \alpha \Z \times \beta \Z$, $(\alpha \beta)^{-1} \in \N$, such that \eqref{eq_condA} is fulfilled. Then the optimal frame bounds of the Gabor system $\G \left( g, \L_{(\alpha, \beta)} \right)$ are given by
	\begin{align}
		A & = \essinf_{(x,\omega) \in \R^2} \, (\alpha \beta) ^{-1} \sum_{k,l \in \Z} V_g g\left( \tfrac{k}{\beta}, \tfrac{l}{\alpha} \right) \, e^{2 \pi i \, (k x + l \omega)}\\
		B & = \esssup_{(x,\omega) \in \R^2} \, (\alpha \beta) ^{-1} \sum_{k,l \in \Z} V_g g\left( \tfrac{k}{\beta}, \tfrac{l}{\alpha} \right) \, e^{2 \pi i \, (k x + l \omega)}
	\end{align}
\end{proposition}
A reformulation  of Proposition \ref{pro_Janssen}, using symmetric time-frequency shifts and the ambiguity function, as well as a generalization to lattices in higher dimensions, can be found in \cite{Faulhuber_Note_2018}.

\subsection{A Remark on Condition A}\label{sec_Cond_A}
We will now make a small excursion, showing that finding a good upper frame bound (not necessarily the least upper frame bound) already implies a good lower frame bound and, hence, a good condition number. Recall that the frame operator acts on functions by the rule
\begin{equation}
	S_{g, \L} f = \sum_{\l \in \L} \langle f, \pi(\l) g \rangle \, \pi(\l) g.
\end{equation}
The following result shows how to compute the distance between the identity operator $I_{L^2}$ and the frame operator using \eqref{eq_condA}. We have
\begin{equation}
	\norm{I_{L^2} - \vol(\L) S_{g,\L}}_{op} \leq \sum_{\l^\circ \in \L^\circ \backslash \{0\}} \left| V_g g \left( \l^\circ \right) \right|.
\end{equation}
This follows from Janssen's representation of the Gabor frame operator
\begin{equation}\label{eq_Janssen_rep}
	S_{g,\L} = \vol(\L)^{-1} \sum_{\l^\circ \in \L^\circ} \langle g, \pi(\l^\circ) g \rangle \, \pi(\l^\circ),
\end{equation}
which is basically a consequence of the so-called fundamental identity of Gabor analysis
\begin{equation}
	\sum_{\l \in \L} \langle f_1, \pi(\l) g_1 \rangle \, \langle \pi(\l) f_2, g_2 \rangle = \vol(\L)^{-1}
	\sum_{\l^\circ \in \L^\circ} \langle \pi(\l^\circ) f_1, g_2 \rangle \, \langle f_2, \pi(\l^\circ) g_1 \rangle.
\end{equation}
For more details we refer to \cite{FeiLue06}, \cite{FeiZim98}, \cite{Gro01}, \cite{Jan95}.

We will now assume that $g$ is normalized, i.e., $\norm{g}_2 = 1$. Also, and only for the rest of this section, we will normalize the frame inequality in the following sense;
\begin{equation}\label{eq_frame_normalized}
	A \norm{f}_2^2 \leq \vol(\L) \sum_{\l \in \L} |V_g f(\l)|^2 \leq B \norm{f}_2^2.
\end{equation}
The bounds in \eqref{eq_frame_normalized} fulfill (see e.g.~\cite[Thm.~5.1.]{JakLem_Density_2016})
\begin{equation}\label{eq_constants}
	0 \leq A \leq 1 \leq B < \infty.
\end{equation}
Note that we included the case where the frame operator is not invertible. The upper bound $B$ is finite as we assume \eqref{eq_condA}. In fact, the results in the work of Tolimieri and Orr \cite{TolOrr95} show that
	\begin{equation}
		\widetilde{B} = \sum_{\l^\circ \in \L^\circ} |V_g g(\l^\circ)|
	\end{equation}	
always is an upper bound, not necessarily optimal, for the normalized frame inequality \eqref{eq_frame_normalized}\footnote{Hence, we have that $\vol(\L)^{-1} \sum_{\l^\circ \in \L^\circ} |V_g g(\l^\circ)|$ always is an upper bound for the frame inequality $\eqref{eq_frame}$.}.

We get the following results, which are a direct consequence of \eqref{eq_Janssen_rep} and are also stated in \cite[Chap.~3.1]{Wie13}. There they are referred to as Janssen's test, which was already introduced in \cite{Tschurt_master}.
\begin{proposition}
	Assume that $g \in \Lt[]$ with $\norm{g}_2 = 1$ and fulfills \eqref{eq_condA}. If
	\begin{equation}
		\widetilde{B} < 2,
	\end{equation}
	then $\G(g,\L)$ is a frame. Also, if $\G(g,\L)$ is not a frame, then
	\begin{equation}
		\widetilde{B} \geq 2.
	\end{equation}
\end{proposition}
\begin{proof}
	We use Janssen's representation of the Gabor frame operator and re-write it as
	\begin{equation}
		\vol(\L) S_{g,\L} = \sum_{\l^\circ \in \L^\circ} \langle g, \pi(\l^\circ) g \rangle \, \pi(\l^\circ) =
		\langle g, \pi(0) g \rangle \, \pi(0) + \sum_{\l^\circ \in \L^\circ \backslash \{0\}} \langle g, \pi(\l^\circ) g \rangle \, \pi(\l^\circ).
	\end{equation}
	Now, $\pi(0) = I_{L^2}$ and $\langle g, \pi(0)g \rangle = \norm{g}_2^2 = 1$. Therefore,
	\begin{equation}
		\norm{\vol(\L) S_{g,\L}}_{op} \leq \norm{I_{L^2}}_{op} + (\widetilde{B} - 1).
	\end{equation}		
	It follows that, if $\widetilde{B} < 2$, then
	\begin{equation}
		\norm{I_{L^2} - \vol(\L) S_{g,\L}}_{op} < \widetilde{B} - 1 < 1,
	\end{equation}
	hence, $\vol(\L) S_{g,\L}$ is invertible and its inverse can be expressed by a Neumann series. The second statement now follows trivially.
\end{proof}
Also, we get an estimate on the lower bound from \eqref{eq_condA} (see also \cite[Chap.~3.1]{Wie13}).
\begin{proposition}
	Assume that $g \in \Lt[]$ with $\norm{g}_2 = 1$ and fulfills \eqref{eq_condA}. Then we have the following estimates for the constants in \eqref{eq_frame_normalized}.
	\begin{equation}
		1 - A \leq \widetilde{B} - 1 \qquad \textnormal{ and } \qquad B \leq \widetilde{B}.
	\end{equation}
\end{proposition}
\begin{proof}
	The statement that $B \leq \widetilde{B}$ was proved by Tolimieri and Orr \cite{TolOrr95}. So, we only prove the statement for the lower frame bound.
		
	If $\widetilde{B} \geq 2$, then the statement is trivially true due to \eqref{eq_constants}. Hence, assume $\widetilde{B} < 2$, which already implies that we have a frame. This also means that
	\begin{equation}
		\norm{I_{L^2} - \vol(\L) S_{g,\L}}_{op} \leq \widetilde{B} - 1 < 1.
	\end{equation}
	Therefore, we have
	\begin{equation}
		A^{-1} = \norm{\left(\vol(\L) S_{g,\L}\right)^{-1}}_{op} \leq \sum_{n \in \N} (\widetilde{B} - 1)^n = \frac{1}{1-(\widetilde{B} - 1)}.
	\end{equation}
	Hence, the result follows.
\end{proof}
The previous result particularly shows that, if the bound $\widetilde{B}$ is close to 1, then the lower bound has to be close to 1 as well. This means, that a good upper frame bound has a good lower frame bound and a good condition number as a consequence. This already pinpoints into the direction of our main result, which is nonetheless stronger under its more restrictive assumptions.

\subsection{Gaussian Windows}
We will now use Proposition \ref{pro_Janssen} to compute sharp frame bounds for the standard Gaussian window $g_0(t) = 2^{1/4} e^{-\pi t^2}$, where the factor in front is chosen to normalize the Gaussian, i.e., $\norm{g_0}_2 = 1$. We start with noting that the standard Gaussian is invariant under the Fourier transform \cite[App.~A]{Fol89}, \cite[Chap.~1.5]{Gro01}, i.e.,
\begin{equation}
	\F g_0 (\omega) = g_0 (\omega).
\end{equation}
This fact, and variations of it involving dilations, will be used implicitly several times in the sequel, especially when using the Poisson summation formula. Furthermore, we note that the Gabor system studied by Gabor \cite{Gab46} was (up to scaling) $\G(g_0,\Z^2)$. Now, we know that this system is complete \cite{BarButGirKla71}, \cite{Per71} (see also \cite{GroHaiRom16}), but just fails to be a frame. Due to the results of Lyubarskii \cite{Lyu92} and Seip and Wallsten \cite{Sei92}, \cite{Sei92_1}, \cite{SeiWal92}, we know that the Gabor system $\G(g, \L)$ is a frame if and only if the lattice $\L$ has density $\delta(\L) > 1$\footnote{This result holds more generally for any (relatively separated) point set with lower Beurling density strictly greater than 1.}.

Another well-known fact about $g_0$ \cite[Chap.~1.5]{Gro01} (and also a nice exercise to verify) is that
\begin{equation}\label{eq_Vg0g0}
	V_{g_0}g_0 (x,\omega) = e^{-\pi i x \omega} e^{-\tfrac{\pi}{2} (x^2+\omega^2)}.
\end{equation}
By Proposition \ref{pro_Janssen}, it follows that the optimal frame bounds of the Gabor system $\G(g_0, \alpha \Z \times \beta \Z)$, $(\alpha \beta)^{-1} \in \N$ are given by
\begin{align}
	A(\alpha, \beta) & = (\alpha \beta) ^{-1} \sum_{k,l \in \Z} (-1)^{\tfrac{k l}{\alpha \beta}}(-1)^{k+l} e^{-\tfrac{\pi}{2} \left( \tfrac{k^2}{\beta^2} + \tfrac{l^2}{\alpha^2} \right)},\\
	B(\alpha, \beta) & = (\alpha \beta) ^{-1} \sum_{k,l \in \Z} (-1)^{\tfrac{k l}{\alpha \beta}}e^{-\tfrac{\pi}{2} \left( \tfrac{k^2}{\beta^2} + \tfrac{l^2}{\alpha^2} \right)}.
\end{align}
This was already computed by Janssen \cite[Sec.~6]{Jan96}.

We note that for $(\alpha \beta)^{-1} \in 2 \N$ the alternating sign, given by $(-1)^{\tfrac{k l}{\alpha \beta}}$, equals $+1$ for all $k,l \in \Z$. We note that in this case the double series splits into a product of two simple series of same type;
\begin{align}
	A(\alpha, \beta) & = (\alpha \beta) ^{-1} \left( \sum_{k \in \Z} (-1)^{k}  e^{-\tfrac{\pi}{2} \tfrac{k^2}{\beta^2}} \right) \,
	\left( \sum_{l \in \Z} (-1)^{l} e^{-\tfrac{\pi}{2} \tfrac{l^2}{\alpha^2}} \right),\\
	B(\alpha, \beta) & = (\alpha \beta) ^{-1} \left( \sum_{k \in \Z} e^{-\tfrac{\pi}{2} \tfrac{k^2}{\beta^2}} \right) \,
	\left( \sum_{l \in \Z} e^{-\tfrac{\pi}{2} \tfrac{l^2}{\alpha^2}} \right).
\end{align}
The expert reader will already draw the connection to theta functions which will be discussed in the next section.

\section{Some Tools from Analytic Number Theory}\label{sec_NT}
From here on, our analysis will be for arbitrary density of the lattice, i.e., $(\alpha \beta)^{-1} \in \R_+$. However, only if the density is an even, positive integer our formulas exactly describe the frame bounds of the Gaussian Gabor system. We will keep that fact in mind, even though we will not mention it explicitly.

\subsection{Theta Functions}\label{sec_theta}
We will recall how Conjecture \ref{con_generalized_Floch} was solved for even densities of the lattice by Faulhuber and Steinerberger \cite{FaulhuberSteinerberger_Theta_2017}. First, we introduce the functions $\theta_2$, $\theta_3$ and $\theta_4$. The literature on theta functions is extensive and any attempt to pick representative textbooks or surveys is doomed. We name the textbooks by Stein and Shakarchi \cite{SteSha_Complex_03} and Whittaker and Watson \cite{WhiWat69} as references.

As customary, we denote the complex plane by $\C$ and the upper half plane by $\H$. A theta function is a function of two complex variables $(z, \tau) \in \C \times \H$ and the classical theta function is
\begin{equation}\label{eq_Theta}
	\Theta(z, \tau) = \sum_{k \in \Z} e^{\pi i k^2 \tau} e^{2 \pi i k z}.
\end{equation}
In this work we will deal with the following list of theta functions which can be expressed by $\Theta$ in one way or another. We define
\begin{align}
	\vartheta_1(z,\tau) & = \sum_{k \in \Z} (-1)^{(k-1/2)} e^{\pi i (k+1/2)^2 \tau} e^{(2k+1)\pi i z},
	& &
	\vartheta_2(z,\tau) = \sum_{k \in \Z} e^{\pi i (k+1/2)^2 \tau} e^{(2k+1)\pi i z},\\
	\vartheta_3(z,\tau) & = \sum_{k \in \Z} e^{\pi i k^2 \tau} e^{2 k \pi i z},
	& &
	\vartheta_4(z,\tau) = \sum_{k \in \Z} (-1)^k e^{\pi i k^2 \tau} e^{2 k \pi i z}.
\end{align}
These functions are Jacobi's classical theta functions.

For our purposes, we will use Jacobi's theta functions where we fix the first argument to be $z = 0$ (so called ``theta nulls"). Hence, for $\tau \in \H$ we define the following theta functions.
\begin{align}
	\theta_2(\tau) & = \vartheta_2(0,\tau) = \sum_{k \in \Z} e^{\pi i \left(k + \tfrac{1}{2} \right)^2 \tau}\\
	\theta_3(\tau) & = \vartheta_3(0,\tau) = \sum_{k \in \Z} e^{\pi i k^2 \tau}\\
	\theta_4(\tau) & = \vartheta_4(0,\tau) = \sum_{k \in \Z} (-1)^k e^{\pi i k^2 \tau}
\end{align}
We note that $\theta_1(\tau) = \vartheta_1(0,\tau) = 0$ for all $\tau \in \H$. Any of the above functions also has a product representation, called the Jacobi triple product representation;
\begin{align}
	\theta_2(\tau) & = 2 \, e^{\tfrac{\pi i}{4} \tau} \prod_{k \in \N} \left(1 - e^{2 k \pi i \tau}\right)\left(1 + e^{2 k \pi i \tau}\right)^2\\
	\theta_3(\tau) & = \prod_{k \in \N} \left(1 - e^{2 k \pi i \tau}\right)\left(1 + e^{(2 k - 1) \pi i \tau}\right)^2\\
	\theta_4(\tau) & = \prod_{k \in \N} \left(1 - e^{2 k \pi i \tau}\right)\left(1 - e^{(2 k - 1) \pi i \tau}\right)^2
\end{align}
Later, we will restrict our attention to purely imaginary $\tau$ and note, that the above functions are real-valued and positive in this case.

It is a well-known fact and easily proved by using the Poisson summation formula (see Appendix \ref{sec_symplectic}) that
\begin{equation}\label{eq_t2_t4}
	\theta_2(\tau) = \sqrt{\tfrac{i}{\tau}} \, \theta_4(-\tfrac{1}{\tau}) \quad \textnormal{ and } \quad \theta_4(\tau) = \sqrt{\tfrac{i}{\tau}} \, \theta_2(-\tfrac{1}{\tau})
\end{equation}
as well as
\begin{equation}\label{eq_t3}
	\theta_3(\tau) = \sqrt{\tfrac{i}{\tau}} \, \theta_3(-\tfrac{1}{\tau}).
\end{equation}

We note that we can write the lower and upper bound of the Gaussian Gabor system $\G(g_0, \alpha \Z \times \beta \Z)$, $(\alpha \beta)^{-1} \in 2 \N$ as products of the above theta functions ($\theta_2$ and $\theta_4$ are always exchangeable by equation \eqref{eq_t2_t4}). 
\begin{align}
	A(\alpha, \beta) & = (\alpha \beta)^{-1} \, \theta_4 \Big( \tfrac{i}{2 \alpha^2} \Big) \, \theta_4 \Big( \tfrac{i}{2 \beta^2} \Big) \label{eq_boundA_theta} \\
	B(\alpha, \beta) & = (\alpha \beta)^{-1} \, \theta_3 \Big( \tfrac{i}{2 \alpha^2} \Big) \, \theta_3 \Big( \tfrac{i}{2 \beta^2} \Big) \label{eq_boundB_theta}
\end{align}
As already observed in \cite{FaulhuberSteinerberger_Theta_2017}, since the product $(\alpha \beta)^{-1}$ is fixed, the problem of finding extremal frame bounds can be reformulated as a problem of a (fixed) parameter and one variable. In \cite{FaulhuberSteinerberger_Theta_2017}, the following result is proved.
\begin{theorem}\label{thm_FS}
	For any fixed $r \in \R_+$, the function
	\begin{align}
		A_r(y) = \theta_4(i r y^{-1}) \theta_4(i r y), \qquad y \in \R_+
	\end{align}
	is maximal if and only if $y = 1$. Also, the function
	\begin{align}
		B_r(y) = \theta_3(i r y^{-1}) \theta_3(i r y), \qquad y \in \R_+
	\end{align}
	is minimal if and only if $y = 1$.
\end{theorem}
This theorem is equivalent to saying that the Gabor frame bounds of the Gabor system $\G(g_0, \alpha \Z \times \beta \Z)$ are extremal if and only if $\alpha = \beta$.

A key ingredient in the work of Faulhuber and Steinerberger \cite{FaulhuberSteinerberger_Theta_2017} is to formulate a problem independent of the parameter $r \in \R_+$, hence independent of the density of the lattice. We note that, in \cite{FaulhuberSteinerberger_Theta_2017}, there is no hint in the proof of Theorem \ref{thm_FS} that the minimality of $B_r$ implies the maximality of $A_r$, not even for the case $r = 1$, which we essentially prove in Theorem \ref{thm_main}.

\subsection{Zeta Functions}\label{sec_zeta}
For this section we will identify $\R^2$ with $\C$. The lattice in focus is then $\alpha \Z \times i \, \beta \Z$, $\alpha, \beta \in \R_+$. Also, we will focus our attention on another class of special functions now, namely zeta functions related to Laplace--Beltrami operators on rectangular tori. A rectangular torus can be identified with a rectangular lattice;
\begin{equation}
	\mathbb{T}^2_{(\alpha, \beta)} = \C \big/ (\alpha \Z \times i \, \beta \Z).
\end{equation}

We will now mainly follow the work of Baernstein and Vinson \cite{BaernsteinVinson_Local_1998} and the work of Osgood, Philips and Sarnak \cite{Osgood_Determinants_1988}. For the torus $\mathbb{T}^2_{(\alpha,\beta)}$ we denote its Laplace--Beltrami operator by $\Delta_{(\alpha,\beta)}$. We note that the eigenfunctions of $\Delta_{(\alpha,\beta)}$ are given by
\begin{equation}
	f_{(\alpha,\beta)}^{k,l}(z) = e^{2 \pi i \, Re\left( z \cdot \left(\tfrac{k}{\alpha} + i \tfrac{l}{\beta} \right)\right)}
\end{equation}
with eigenvalues\footnote{We chose the sign of $\Delta_{(\alpha,\beta)}$ such that the eigenvalues are non-negative.}
\begin{equation}
	\l_{(\alpha,\beta)}^{k,l} = \left(2 \pi \left| \tfrac{k}{\alpha} + i \, \tfrac{l}{\beta} \right| \right)^2.
\end{equation}
Consequently, for $t \in \R_+$, the heat kernel is then given by the formula
\begin{equation}
	p_{(\alpha,\beta)}(z;t) = \sum_{k,l \in \Z} e^{-t \l_{(\alpha,\beta)}^{k,l}} f_{(\alpha,\beta)}^{k,l}(z).
\end{equation}
At this point, we remark that $(\frac{k}{\alpha} + i \frac{l}{\beta})$ is an element of the dual lattice of $\alpha \Z \times i \, \beta \Z$ and that, in general, the heat kernel on a (complex) torus can be explicitly written in terms of the (complex) dual lattice $\L^\perp$;
\begin{equation}
	p_\L(z;t) = \sum_{\l^\perp \in \L^\perp} e^{-4 \pi^2 t |\l^\perp|^2} e^{2 \pi i z \cdot \l^\perp}.
\end{equation}

Getting back to the family of tori $\T_{(\alpha, \beta)}^2$, the trace of the heat kernel associated to $\Delta_{(\alpha,\beta)}$ is given by
\begin{equation}
	\textnormal{tr}(p_{(\alpha,\beta)})(t) = \sum_{k,l \in \Z} e^{-t \l_{(\alpha,\beta)}^{k,l}} = \sum_{k,l \in \Z} e^{-4 \pi^2 t \left( \tfrac{k^2}{\alpha^2} + \tfrac{l^2}{\beta^2} \right)}
\end{equation}
and its zeta function is given by
\begin{align}\label{eq_zeta}
	Z_{(\alpha,\beta)}(s) & = \sump{k,l \in \Z} \left(\l_{(\alpha,\beta)}^{k,l}\right)^{-s} = (2 \pi)^{-2 s} \sump{k,l \in \Z} \left| \tfrac{k}{\alpha} + i \, \tfrac{l}{\beta} \right|^{-2s}\\
	& =	(2 \pi)^{-2 s} (\alpha \beta)^{s} \sump{k,l \in \Z} \frac{\left(\tfrac{\alpha}{\beta}\right)^s}{|k + i \, \tfrac{\alpha}{\beta} l|^{2s}} \, ,
\end{align}
where the prime indicates that the sum excludes the origin. The determinant of the Laplace--Beltrami operator is formally given as the product of the non-zero eigenvalues,
\begin{equation}
	\det \Delta_{(\alpha,\beta)} = \prod_{\l_{(\alpha,\beta)}^{k,l} \neq 0} \l_{(\alpha,\beta)}^{k,l}.
\end{equation}
This product is not always meaningful and, usually, one uses the zeta regularized determinant of the Laplace--Beltrami operator $\Delta_{(\alpha,\beta)}$ instead, which is given by
\begin{equation}
	\detp \, \Delta_{(\alpha,\beta)} = e^{- \tfrac{d}{ds} Z_{(\alpha,\beta)} \big|_{s=0}}.
\end{equation}
The problem under consideration is to find the pair $(\alpha,\beta)$ which maximizes the determinant $\detp \, \Delta_{(\alpha,\beta)}$ for fixed volume of the torus, i.e., $0 < \alpha \beta$ fixed. We note that the last line in equation \eqref{eq_zeta} is, up to the factor in front, a real analytic Eisenstein series. In general, for $\tau = x + i y \in \mathbb{H}$ and $Re(s) > 1$, the real analytic Eisenstein series is given by
\begin{equation}
	E(\tau,s) = \sump{k,l \in \Z} \frac{y^s}{|k + l \tau|^{2s}}.
\end{equation}
It can be continued analytically for $s \in \C$ with a pole at $s = 1$ with residue $\pi$, which follows from Kronecker's limit formula given below. In order to compute the regularized determinant $\detp \Delta_{(\alpha,\beta)}$, Osgood, Philips and Sarnak employ Kronecker's limit formula \cite[eq.~(4.4)]{Osgood_Determinants_1988}
\begin{align}
	E(\tau,0) & = -1\\
	\frac{\partial}{\partial s} E \Big|_{s=0} & = - 2 \log\left( 2 \pi \, y^{1/2} |\eta(\tau)|^2\right).
\end{align}
Here, $\eta$ is the Dedekind eta function which we will study in more detail in the upcoming section.

By the last results, it follows that
\begin{align}
	\frac{d}{ds} Z_{(\alpha, \beta)} \Big|_{s=0} & = -2 \log \left( \tfrac{2 \pi}{(\alpha \beta)^{1/2}} \right) E\left( i \, \tfrac{\alpha}{\beta},0\right) + \frac{\partial}{\partial s} E \Big|_{s=0}\\
	& = -\log(\alpha \beta) - \log \left( \left( \tfrac{\alpha}{\beta} \right) \left| \eta \left( i \, \tfrac{\alpha}{\beta}\right) \right|^4 \right).
\end{align}
Therefore, the determinant is given by
\begin{equation}\label{eq_det_eta}
	\detp \, \Delta_{(\alpha,\beta)} = (\alpha \beta ) \left( \tfrac{\alpha}{\beta} \right) \left| \eta \left( i \, \tfrac{\alpha}{\beta}\right) \right|^4.
\end{equation}
Since we fixed the product $\alpha \beta$, which expresses the surface area of the torus, we see that the problem of maximizing the determinant is invariant under scaling and we may therefore assume that $\alpha \beta = 1$.

For the rest of this work, we choose the following normalization and notation for our torus;
\begin{equation}
	\mathbb{T}^2_\alpha = \C \Big/ \left(\alpha^{1/2} \Z \times i \, \alpha^{-1/2} \Z \right).
\end{equation}
Likewise, in the sequel we will use the index $\alpha$ for the index pair $(\alpha^{1/2}, \alpha^{-1/2})$, appearing in the Laplace--Beltrami operator, the heat kernel, the determinant and the zeta function.

For general tori, it was shown in \cite{Osgood_Determinants_1988} that the hexagonal torus is the unique maximizer of the determinant by using a proof which combines analytic and numerical methods. For rectangular tori, it was shown in \cite{Faulhuber_Determinants_2018} that the square torus is the unique maximizer by purely analytic means. It also follows from the results in Montgomery's work \cite{Montgomery_Theta_1988} that the hexagonal torus maximizes the determinant among all tori of given surface area. That was also remarked in the work of Baernstein and Vinson \cite{BaernsteinVinson_Local_1998}. Likewise, it follows from one of the results of Faulhuber and Steinerberger \cite{FaulhuberSteinerberger_Theta_2017} that the square torus is the maximizer of the determinant among all rectangular tori of given surface area.

We will shortly sketch the arguments used in \cite{BaernsteinVinson_Local_1998} and \cite{Montgomery_Theta_1988}. For brevity, we set
\begin{equation}
	W_\alpha(t) = \textnormal{tr} \left( p_\alpha \right)(t) = \sum_{k,l \in \Z} e^{-4 \pi^2 t \left( \tfrac{k^2}{\alpha} + \alpha \, l^2 \right)}.
\end{equation}
The results in \cite{FaulhuberSteinerberger_Theta_2017} and \cite{Montgomery_Theta_1988} show that
\begin{equation}\label{eq_Wa_W1}
	W_\alpha(t) \geq W_1(t), \qquad \forall t \in \R_+
\end{equation}
with equality if and only if $\alpha = 1$.

We will show how this result already implies the result for the determinant. Now, the zeta function also has an integral representation and it is connected to the heat kernel via the Mellin transform \cite{Jorgenson_Heat_2001}.
\begin{equation}\label{eq_xi}
	\xi_\alpha = (2 \pi)^s \, \Gamma(s) Z_\alpha (s) = \int_{\R_+} (W_\alpha(t) - 1) \, t^{s-1} \, dt, \qquad s \in \C, \, Re(s) > 1,
\end{equation}
where $\Gamma$ denotes the Gamma function. The function $\xi_\alpha$ extends to a meromorphic function with a simple pole of order 1 at $s=1$ \cite{Deu37}. Montgomery \cite{Montgomery_Theta_1988} uses the regularized version
\begin{equation}\label{eq_zeta_regularized}
	(2 \pi)^s \, \Gamma(s) Z_\alpha (s) = \frac{1}{s-1} - \frac{1}{s} +\int_1^\infty \left(W_\alpha(t) - 1\right) \left(t^{s-1} + t^{-s} \right) \, dt,
\end{equation}
which is well defined for $s \in \C$, except for the points 0 and 1, and yields the functional equation
\begin{equation}\label{eq_zeta_functional}
	(2 \pi)^s \, \Gamma(s) Z_\alpha (s) = (2 \pi)^{1-s} \, \Gamma(1-s) Z_\alpha (1-s).
\end{equation}
We repeat now the arguments given in \cite{BaernsteinVinson_Local_1998}. Since $W_\alpha(t)$ does not depend on $s$, by differentiating both sides of \eqref{eq_zeta_regularized}, it follows that
\begin{equation}
	Z'_\alpha (s) \geq Z'_1(s), \qquad s > 1,
\end{equation}
where the prime indicates differentiation with respect to $s$. By taking the limit $s \to 1$ and using the functional equation \eqref{eq_zeta_functional}, we get
\begin{equation}
	Z'_\alpha (0) \geq Z'_1 (0)
\end{equation}
and, since $\detp \Delta_\alpha = e^{-Z_\alpha '(0)}$,
\begin{equation}
	\detp \Delta_\alpha \leq \detp \Delta_1, \qquad \forall \alpha \in \R_+
\end{equation}
with equality if and only if $\alpha = 1$.

\subsection{The Dedekind Eta Function}
The last function we investigate is the Dedekind eta function. We start by defining the function by its product representation;
\begin{equation}\label{eq_Dedekind}
	\eta(\tau) = e^{\pi i \tau/12} \prod_{k \in \N} (1 - e^{2 k \pi i \tau}), \qquad \tau \in \mathbb{H}.
\end{equation}
We note that it can be expressed as a product of theta functions in the following way (see e.g.~\cite[Chap.~3]{Borwein_AGM_1987}, \cite[Chap.~4]{ConSlo98} or \cite[eq.~(21.41) p.~463, Ex.~1 p.~466]{WhiWat69})\footnote{The function $\tfrac{\partial \vartheta_1}{\partial z}(0,\tau)$ is often denoted by $\theta_1'(\tau)$.};
\begin{equation}
	2 \, \eta(\tau)^{3} = \dfrac{\partial}{\partial z} \vartheta_1(z,\tau) \Big|_{z=0} = \vartheta_2(0, \tau) \, \vartheta_3(0, \tau) \, \vartheta_4(0, \tau).
\end{equation}

Since we are interested in rectangular tori, we are only interested in purely imaginary arguments $\tau = i y$, $y \in \R_+$. We note that in this case it follows from formula \eqref{eq_Dedekind} defining $\eta$, that $\eta(i y) \in \R_+$. Also, we have the following connection to the theta functions defined in Section \ref{sec_theta};
\begin{equation}\label{eq_eta_theta}
	2 \, \eta(i y)^3 = \theta_2(i y) \, \theta_3(i y) \, \theta_4(i y).
\end{equation}

\subsection{Proof of Theorem \ref{thm_main}}
We collected all the material to prove our main result. By using the Jacobi identities \eqref{eq_t2_t4} and \eqref{eq_t3}, we write the optimal Gabor frame bounds, of the Gaussian Gabor system $\G \left(g_0, \tfrac{1}{\sqrt{2}} \L_y \right)$ with $\L_y = (y^{1/2} \Z \times y^{-1/2} \Z)$, $y \in \R_+$, given by \eqref{eq_boundA_theta} and \eqref{eq_boundB_theta}, as
\begin{align}
	A \left( \L_y \right) & = 2 \, \theta_4(i y^{-1}) \, \theta_4(i y) = 2 \, y^{1/2} \theta_2(i y) \, \theta_4(i y)\\
	B \left( \L_y \right) & = 2 \, \theta_3(i y^{-1}) \, \theta_3(i y) = 2 \, y^{1/2} \theta_3(i y) \, \theta_3(i y).
\end{align}

We start now by taking both sides of equation \eqref{eq_eta_theta} to the power 4 and then multiply by $y^3$, which gives in combination with \eqref{eq_det_eta}
\begin{equation}\label{eq_det_theta}
	2^4 \, (\detp \Delta_y)^3 = 2^4 \, y^3 \, \eta(i y)^{12} = y^3 \, \theta_2(i y)^4 \theta_3(i y)^4 \theta_4(i y)^4.
\end{equation}
By multiplying both sides of the last equation by $2^6$, it follows that
\begin{equation}\label{eq_det_A_B}
	2^{10} \, (\detp \Delta_y)^3 = \left(2 \, y^{1/2} \theta_2(i y) \theta_4(i y) \right)^4 \, \left(2 \, y^{1/2} \theta_3(i y)^2 \right)^2= A(\L_y)^4 B(\L_y)^2.
\end{equation}
Also, we note that for all $t \in \R_+$ we have
\begin{equation}\label{eq_W_t3}
	W_{y}\left( \tfrac{t}{4\pi} \right) = \sum_{k,l \in \Z} e^{-\pi t \left( \tfrac{k^2}{y} + y \, l^2\right)} = \widetilde{B} \left( \tfrac{t}{\sqrt{2}} \, \L_y \right) = \theta_3(i t \, y^{-1}) \, \theta_3(i t \, y) = y^{1/2} \theta_3(i t \, y)^2,
\end{equation}
which by Theorem \ref{thm_FS} is minimal if and only if $y = 1$. Starting from the minimality of the trace of the heat kernel $W_1(t)$, we obtain the following chain of implications from the results in Section \ref{sec_zeta}, with details provided below;
\begin{align}\label{eq_main}
	W_y(t) \geq W_1(t), \; \forall t \in \R_+
	& \Longrightarrow \detp \Delta_y \leq \detp \Delta_1 \\
	& \Longrightarrow y^{1/2} \, \theta_2(i y) \, \theta_4(i y) \leq \theta_2(i) \, \theta_4(i),
\end{align}
The first implication uses the following assertions from Section \ref{sec_zeta};
\begin{equation}
	W_y(t) \geq W_1(t), \; \forall t \in \R_+
	\Longrightarrow Z'_y(0) \geq Z'_1(0)
	\Longrightarrow \detp \Delta_y \leq \detp \Delta_1 .
\end{equation}
By combining \eqref{eq_det_theta} and \eqref{eq_W_t3}, we get $2^4 \left(\detp \Delta_y \right)^3 = \left(y^2 \, \theta_2(i y)^4 \,\theta_4(i y)^4 \right) \, W_y \left( \tfrac{1}{4\pi} \right)^2$ (with $t = 1$). Therefore, the maximality of the determinant $\detp \Delta_y$ for $y = 1$ implies the maximality of the product $y^2 \, \theta_2(i y)^4 \,\theta_4(i y)^4$ for $y = 1$, as the trace of the heat kernel $W_y(1)$ is minimal in this case. Since the result on the determinant is implied by the result on the heat kernel, the fact that the trace $W_y(t)$ is minimal for $y = 1$ (for all $t \in \R_+$) also implies that the product $y^{1/2} \, \theta_2(i y) \,\theta_4(i y)$ is maximal for $y = 1$.

In the view of Theorem \ref{thm_FS}, we have that $B_r(y) \geq B_r(1)$ for all $r \in \R_+$ implies that $A_1(y) \leq A_1(1)$ (note that the index of $A$ is $r=1$). However, Theorem \ref{thm_FS} also tells us that $A_r(y) \leq A_r(1)$ for all $r \in \R_+$ and this follows independently form the result on $B_r$ (see the proof in \cite{FaulhuberSteinerberger_Theta_2017}).

As a further consequence, by using \eqref{eq_det_A_B}, we get the following chain of implications for the Gabor frame $\G \left(g_0, \tfrac{1}{\sqrt{2}} \L_y \right)$;
\begin{align}
	\widetilde{B} \left( \sqrt{\tfrac{t}{2}} \, \L_y \right) \geq \widetilde{B} \left( \sqrt{\tfrac{t}{2}} \, \L_1 \right), \, \forall t \in \R_+
	\Longrightarrow & \, A \left(\tfrac{1}{\sqrt{2}} \L_y \right)^4 B \left(\tfrac{1}{\sqrt{2}} \L_y \right)^2 \leq A \left(\tfrac{1}{\sqrt{2}} \L_1 \right)^4 B \left(\tfrac{1}{\sqrt{2}} \L_1 \right)^2 \\
	\Longrightarrow & \, A \left(\tfrac{1}{\sqrt{2}} \L_y \right) \leq A \left(\tfrac{1}{\sqrt{2}} \L_1 \right).
\end{align}
For the first implication we used the fact that for $t=1$ we have $2 \widetilde{B}\left( \tfrac{1}{\sqrt{2}} \, \L_y \right) = B \left(\tfrac{1}{\sqrt{2}} \L_y \right)$\footnote{The factor 2 in front of $\widetilde{B}$ is, of course, determined by the density of the Gabor system and the fact that we defined $\widetilde{B}$ in a way to suit the re-normalized results in Section \ref{sec_Cond_A}. More generally, for $n \in \N$, we have $2n \, \widetilde{B}\left( \tfrac{1}{\sqrt{2n}} \, \L_y \right) = B \left(\tfrac{1}{\sqrt{2n}} \L_y \right)$, as already remarked in \cite[Sect.~4]{TolOrr95}.}. The remaining argument is to note that the product $A^4 B^2$ of the Gabor frame bounds yields (up to a constant) the determinant of the Laplace--Beltrami operator. This completes the proof of Theorem \ref{thm_main}. \hfill $\qed$

\section{The Different Aspects of the Proof}\label{sec_discuss}
We note that the proof relies on several curious facts. We will discuss the different aspects in this section and start with an overview.

First of all, the results of Faulhuber and Steinerberger \cite{FaulhuberSteinerberger_Theta_2017} and of Montgomery \cite{Montgomery_Theta_1988} are results about periodized Gaussian functions on a lattice. The parameter appearing in those works should be interpreted as the density of the lattice. However, the results can also be interpreted as problems for a family of heat kernels on a torus with varying metric. Only for very special cases the results describe the exact behavior of Gaussian Gabor frame bounds. Also, we note that we needed the optimality of the Tolimieri and Orr bound for the square lattice for all densities to establish the optimality of the lower frame bound for the square lattice of density 2. All attempts of the author to show that the results are independent of the parameter $t$, hence holding for arbitrary (even) density, ended up to be equivalent to Theorem \ref{thm_FS}. But then the optimality of the lower frame bound follows independently from the optimality of the upper frame bound. Hence, it would be nice to find an argument not relying on \cite{FaulhuberSteinerberger_Theta_2017}, which shows that the frame bound problem is invariant under scaling.

Another curiosity is that Jacobi's $\vartheta_3$ function as well as the Dedekind eta function are actually modular functions of positive weight with respect to $\tau \in \H$. The frame bounds of a Gaussian Gabor frame are also modular forms, however, of weight 0. It is the separability of the problem which makes the proof of our Theorem \ref{thm_main} work. At the moment, it seems as if the techniques used in this work, will not work for non-separable lattices.

But before going into detail, we will shortly explain why density 2 is such a special case.

\subsection{Why Density 2?}
We will now shortly describe one key aspect, why our result was only established for density 2. Recall, that the sharp frame bounds are derived by sampling $V_{g_0}g_0$ on the adjoint time-frequency lattice and that we have, according to formula \eqref{eq_Vg0g0},
\begin{equation}
	V_{g_0}g_0 (x, \omega) = e^{-\pi i x \omega} e^{- \tfrac{\pi}{2} (x^2 + \omega^2)}.
\end{equation}
The adjoint of the square lattice of density 2, $\L = \tfrac{1}{\sqrt{2}} \Z^2$, is given by $\L^\circ = \sqrt{2} \, \Z^2$. This means that for $(k,l) \in \Z^2$ we have
\begin{equation}
	V_{g_0}g_0(\sqrt{2} \, k, \sqrt{2} \, l) = e^{-\pi (k^2 + l^2)}.
\end{equation}
We have that $e^{-\pi (x^2+\omega^2)}$ is a fixed point of the (symplectic) Fourier transform. Hence, Poisson's summation formula is particularly nice in this case and, also, the above formula allows us to easily connect to Jacobi's theta functions and the Dedekind eta function with purely imaginary arguments.

\subsection{A Problem for the Heat Kernel}
The results of Faulhuber and Steinerberger \cite{FaulhuberSteinerberger_Theta_2017} and of Montgomery \cite{Montgomery_Theta_1988} also provide solutions about a problem concerning the temperature on a torus. For a more recent and thorough study on the temperature problem on the torus we refer to \cite{Faulhuber_Rama_2019}.

We identify the torus with a fundamental cell\footnote{The fundamental cell of a lattice is the parallelogram spanned by the basis vectors of $\L$.} of the lattice $\L$, $\vol(\L)$ fixed, and denote it by
\begin{equation}
	\T_\L^2 = \R^2 \big/ \L.
\end{equation}
Recall from Section \ref{sec_zeta} that the explicit formula for the heat kernel on a torus involves the dual lattice. By using the Poisson summation formula, the corresponding heat kernel can be written using the lattice, which defines the torus, rather than its dual;
\begin{equation}
	p_\L (z;t) = \frac{1}{4 \pi t} \sum_{\l \in \L} e^{-\tfrac{\pi}{4 \pi t} |\l + z|^2}, \qquad z \in \T_\L^2.
\end{equation}
By the symplectic version of Poisson's summation formula (see Appendix \ref{sec_symplectic}) we have the alternative representation
\begin{equation}
	\widetilde{p}_\L (z;t) = \vol(\L)^{-1} \sum_{\l^\circ \in \L^\circ} e^{-\pi (4 \pi t) \, |\l^\circ|^2} e^{2 \pi i \sigma(\l^\circ, z)}.
\end{equation}
The two formulas describe, of course, the same temperature distribution, i.e.,
\begin{equation}
	p_\L(z;t) = \widetilde{p}_\L(z;t).
\end{equation}
For time $t > 0$, we denote the minimal and the maximal temperature of the heat kernel by
\begin{align}
	m_\L(t) = \min_{z \in \T_\L^2} \, \widetilde{p}_\L(z;t)
	\qquad \textnormal{ and } \qquad
	M_\L(t) = \max_{z \in \T_\L^2} \, \widetilde{p}_\L(z;t).
\end{align}
By using the triangle inequality, we easily conclude that
\begin{equation}
	\widetilde{p}_\L(z;t) \leq \widetilde{p}_\L (0;t),
\end{equation}
which shows that $M_\L(t)$ is taken at the origin and, due to periodicity, at any other lattice point of $\L$. Also, $M_\L(t)$ equals the trace $W_\L(t)$ of the heat kernel. Locating the minimum is not as easy and the author is not aware of a closed expression for this problem. For separable lattices and for the hexagonal lattice, the minimum is achieved in a ``deep hole". This is where one would naturally expect the minimum as this is the point with the largest distance to its closest neighbors. It is also the point which is covered last if one places disks of radius $\varepsilon$ at the lattice points and blows them up until they cover the whole plane. However, the generic non-separable case seems to be that the minimum is achieved at some point close to a ``deep hole" and that the location also depends on the parameter $t$. Similar observations are described in the article by Baernstein and Vinson \cite{BaernsteinVinson_Local_1998}.

Montgomery's result \cite{Montgomery_Theta_1988} shows that
\begin{equation}
	M_{\L_h}(t) \leq M_\L(t), \qquad \forall t \in \R_+,
\end{equation}
where $\L_h$ is the hexagonal lattice. The results of Faulhuber and Steinerberger \cite{FaulhuberSteinerberger_Theta_2017} show that
\begin{equation}
	m_1(t) \geq m_\alpha (t) \qquad \textnormal{ and } \qquad M_1 (t) \leq M_\alpha (t), \qquad \forall t \in \R_+,
\end{equation}
where we used the notation for separable tori from Section \ref{sec_zeta}. The problem remaining open, which is also addressed in \cite[Sec.~6, eq.~(6.3)]{Faulhuber_Rama_2019}, is whether
\begin{equation}
	m_{\L_h}(t) \stackrel{(?)}{\geq} m_\L(t), \qquad \forall t \in \R_+,
\end{equation}
which is closely related to Conjecture \ref{con_generalized_Strohmer_Beaver}.

We note the close connection of the above problems to the Strohmer and Beaver Conjecture. The key observation is that the time variable $t$ can also be used to describe the volume of the lattice. Consider the Gaussian Gabor system
\begin{equation}
	\G (g_0, (8 \pi t)^{-1/2} \L)
\end{equation}
with $\vol(\L) = 1$. This means that the time-frequency shifts are carried out along a scaled version of the lattice $\L$ with density $8 \pi t$. We note that we only have a frame if $8 \pi t > 1$ and for $4 \pi t = n \in \N$, the sharp Gaussian Gabor frame bounds are given by \footnote{For $t = \tfrac{n}{8 \pi}$, $n \in \N$, the time-frequency lattice has density $n$, which means that $4 \pi t \in \N$ gives a time-frequency lattice of even density.}
\begin{equation}
	A\left((8 \pi t)^{-1/2} \L\right) = 8 \pi t \, m_\L (t) \Big|_{t = \tfrac{n}{4 \pi}} = 2 n \, m_\L \left( \tfrac{n}{4 \pi} \right)
\end{equation}
and
\begin{equation}
	B\left((8 \pi t)^{-1/2} \L\right) = 8 \pi t \, M_\L (t) \Big|_{t = \tfrac{n}{4 \pi}} = 2 n \, M_\L \left( \tfrac{n}{4 \pi} \right)
\end{equation}
More general, $M_\L \left( \tfrac{t}{8 \pi} \right) = \widetilde{B} \left( (8 \pi t)^{-1/2} \L \right)$ gives the Tolimieri and Orr bound for the Gaussian Gabor system with lattice density $8 \pi t$ and (recall Section \Ref{sec_Cond_A})
\begin{equation}
	B\left((8 \pi t)^{-1/2} \, \L\right) \leq 8 \pi t \, \widetilde{B}\left((8 \pi t)^{-1/2} \, \L\right) = 8 \pi t \, M_\L \left( \tfrac{t}{8 \pi} \right).
\end{equation}
We have equality whenever $4 \pi t \in \N$.

One curios part of the proof is now that we lose the parameter $t$ when moving from the trace of the heat kernel to the zeta function, as we integrate over $t \in \R_+$. Actually, we need the information that the minimal trace of the heat kernel stays minimal for all $t$ in order to establish our main result, which then implies the maximality result for $t=\tfrac{1}{4 \pi}$, which corresponds to the Gabor case of density 2. Lastly, we note that the (real) parameter $t$ is usually interpreted as ``time" in the case of the heat kernel, whereas the density of the lattice (which equals the reciprocal value of the surface area of the torus) is fixed. However, it can also be interpreted as the density of the lattice for fixed time or as the product of both, density and time.

\subsection{The Modular Aspect of the Proof}\label{sec_modular}
We will only give some basic details on modular functions and refer to the textbooks of Serre \cite{Serre_Course_1973} or Stein and Shakarchi \cite{SteSha_Complex_03}. A modular function $f$ of weight $k$ satisfies
\begin{equation}
	f(\tau) = (c \tau + d)^{-k} f \left(\tfrac{a \tau + b}{c \tau +d}\right), \qquad \S \in PSL(2,\Z) = SL(2,\Z)\slash\{\pm I\}.
\end{equation}
We note that $\Z^2$ is invariant under the action of the group $SL(2,\Z)$, which consists of determinant 1 matrices with integer entries. Expressed differently, we only change the basis of the lattice under the action of this group. Since, in general, the matrices $S$ and $-S$ generate the same lattice, they are identified which is why we only consider $PSL(2,\Z)$, the modular group.

Next, we note that a 2-dimensional lattice can be identified (modulo rotation) with a complex number $\tau = x + iy \in \H$. The generating matrix then has the form
\begin{equation}
	S = y^{-1/2} \left(
	\begin{array}{c c}
		1 & x\\
		0 & y
	\end{array}
	\right).
\end{equation}
We set
\begin{equation}
	a(\tau;t) = m_\L \left( \tfrac{t}{4 \pi} \right)
	\qquad \textnormal{ and } \qquad
	b(\tau;t) = M_\L \left( \tfrac{t}{4 \pi} \right),
\end{equation}	
where we identify $\tau = x + i y$ with $\L$. Then, it follows that $a(\tau;t)$ and $b(\tau;t)$ are modular forms of weight 0 with respect to $\tau$. After all, choosing a different basis for the lattice just changes the order of summation. However, the values do not depend on the order of summation as the sums defining the heat kernel are absolutely convergent. In the separable case, $\tau$ is purely imaginary, i.e., $\tau = i y$ with $y \in \R_+$, and the values of $a(i y;t)$ and $b(i y;t)$ split into products of theta functions of the form
\begin{align}
	a(iy;t) = \theta_4 (i t y^{-1}) \, \theta_4 (i t y) \quad \textnormal{ and } \quad
	b(iy;t) = \theta_3 (i t y^{-1}) \, \theta_3 (i t y).
\end{align}
In the non-separable case, the author is not aware of a possibility to write $a(\tau;t)$ and $b(\tau;t)$ as just a product of ``nulls" of Jacobi's theta functions. However, the possibility to write $a$ and $b$ as products of theta functions lastly gives the connection to the Dedekind eta function, which is a modular form of weight 1/2.

We recall from \cite{Osgood_Determinants_1988}, or by repeating the calculations in Section \ref{sec_zeta} for general $\tau \in \H$, that the determinant of the Laplace--Beltrami operator on a (general) torus $\T^2_{\L_\tau}$ is given by
\begin{equation}
	\detp \Delta_\tau = Im(\tau) \, \eta(\tau)^2 \, \overline{\eta(\tau)}{\,}^2,
\end{equation}
where $\T^2_{\L_\tau} = \C / \L_\tau$ with $\L_\tau = Im(\tau)^{-1/2}(\Z \times \tau \Z)$, $\tau \in \H$ (also compare with \cite[Chap.~VII]{Serre_Course_1973}). The behavior of the eta function under the modular group is given by
\begin{equation}
	\eta(\tau) = (c \tau + d)^{-1/2} \eta \left( \tfrac{a \tau + b}{c \tau + d} \right)
\end{equation}
and the imaginary part of $\tau$ transforms as
\begin{equation}
	Im(\tau) = |c \tau + d|^2 Im\left( \tfrac{a \tau + b}{c \tau + d} \right).
\end{equation}
Hence, we have the result
\begin{equation}
	\detp \Delta_\tau = \detp \Delta_{\tfrac{a \tau + b}{c \tau + d}} \, .
\end{equation}
It follows that the determinant of the Laplace--Beltrami operator on a two dimensional torus is a modular form of weight 0 with respect to the metric on the torus.

By using the symplectic version of Poisson's summation formula we find that the determinant also behaves like a modular form with respect to the volume of the lattice (also see Appendix \ref{sec_symplectic})
\begin{equation}
	\detp \Delta_\L = \vol(\L)^{-1} \detp \Delta_{\L^\circ} = \vol(\L)^{-1} \detp \Delta_{\vol(\L)^{-1} \L} \, .
\end{equation}
However, the volume is a positive number and the behavior of the determinant is therefore comparable to the behavior of the Dedekind eta function for purely imaginary arguments. As already explained, the parameter $t$ describes the volume of the lattice. It is for this reason that we cannot really distinguish between the modular parameter $\tau \in \H$ and the real parameter $t$ for separable tori which causes one of the curiosities mentioned at the beginning of this section.

\begin{appendix}
\section{The Symplectic Poisson Summation Formula}\label{sec_symplectic}
When working in the time-frequency plane or phase space, it is often advantageous to exploit the symplectic structure of $\R^{2d}$. The concept of phase space allows for a simultaneous treatise of two aspects of a function, coupled via the Fourier transform, such as position and momentum or, in the case of Gabor analysis, time and frequency. For more details on phase space methods and symplectic geometry we refer to the textbooks of Folland \cite{Fol89} and de Gosson \cite{Gos06, Gos11}.

For two points $z = (x,\omega)$ and $z' = (x',\omega')$ in phase space, the symplectic form is given by
\begin{equation}
	\sigma(z,z') = x \cdot \omega' - \omega \cdot x' = z \cdot J z', \quad
	J = 
	\left(
		\begin{array}{r l}
			0 & I\\
			-I & 0
		\end{array}
	\right)	
\end{equation}
where $\cdot$ denotes the standard Euclidean inner product of two vectors in $\R^d$. A matrix $S$ is symplectic if and only if
\begin{equation}
	\sigma( S z, S z') = \sigma(z,z'), \quad \textnormal{ or equivalently, } \quad S^T J S = J.
\end{equation}
We note that $\sigma(z,z') = - \sigma(z',z)$ and, hence, $\sigma(z,z) = 0$ for all $z$. The group of all symplectic matrices is denoted by $Sp(d)$ and is a proper subgroup of $SL(2d, \R)$, except for the case $d = 1$, where we have $Sp(1) = SL(2,\R)$.

The classical version of Poisson's summation formula for lattices in $\Rd$ is 
\begin{equation}
	\sum_{\l \in \L} f(\l+z) = \vol(\L)^{-1} \sum_{\l^\bot \in \L^\bot} \F f(\l^\bot) e^{2 \pi i \l^\bot \cdot z}, \qquad z \in \Rd.
\end{equation}
The dual of the lattice $\L = M \Z^d$ is given by $\L^\bot = M^{-T} \Z^d$, $M \in GL(d,\R)$. The details when this formula holds pointwise have been worked out in \cite{Gro_Poisson_1996}.

However, we want to work in phase space which is always even dimensional. For a lattice $\L = M \Z^{2d}$ we can replace the dual lattice by the adjoint lattice $\L^\circ = J M^{-T} \Z^{2d}$, the Fourier transform by the symplectic Fourier transform and the Euclidean inner product in the exponent is replaced by the symplectic form. The symplectic version of Poisson's summation formula is given by (see also \cite{Faulhuber_Note_2018} as well as \cite{FeiLue06})
\begin{equation}
	\sum_{\l \in \L} f(\l+z) = \vol(\L)^{-1} \sum_{\l^\circ \in \L^\circ} \F_\sigma f(\l^\circ) \, e^{2 \pi i \sigma(\l^\circ, z)}, \qquad z \in \R^{2d},
\end{equation}
where $\F_\sigma$ is the symplectic Fourier transform of a function of $2d$ variables, given by
\begin{equation}
	\F_\sigma f(z) = \int_{\R^{2d}} f(z') e^{-2 \pi i \sigma(z,z')} \, dz'.
\end{equation}
We call a lattice symplectic if its generating matrix is a multiple of a symplectic matrix. From the definition of a symplectic matrix, it follows that in this case $\L^\circ = \vol(\L)^{-1/d} \L$, as by definition $J S^{-T} J^{-1} \Z^{2d} = S \Z^{2d}$ for $S \in Sp(d)$ and $J^{-1} \Z^{2d} = \Z^{2d}$. We note that any 2-dimensional lattice is symplectic.

Also, using the symplectic Fourier transform has the advantage that all $2$-dimensional, normalized Gaussians are eigenfunctions of the symplectic Fourier transform $\F_\sigma$ with eigenvalue 1 \cite{Faulhuber_Note_2018}. Hence, for 2-dimensional Gaussians and 2-dimensional lattices all formulas become particularly nice.

\end{appendix}


\end{document}